\documentclass[a4,12pt]{amsart}
\usepackage{amssymb,amstext,amsmath,amscd,amsthm,amsfonts,enumerate,latexsym}
\usepackage{color}
\usepackage[dvipdfmx]{graphicx}
\usepackage[all]{xy}

\usepackage[dvipdfmx]{hyperref}
\usepackage{ytableau}
\theoremstyle{plain}
\newtheorem{thm}{Theorem}[section]

\newtheorem*{thm*}{Theorem}
\newtheorem*{cor*}{Corollary}

\newtheorem{prop}[thm]{Proposition}

\newtheorem{lem}[thm]{Lemma}
\newtheorem{cor}[thm]{Corollary}

\newtheorem{claim}{Claim}
\newtheorem*{claim*}{Claim}

\theoremstyle{definition}

\newtheorem{definition}[thm]{Definition}
\newtheorem{ex}[thm]{Example}

\newtheorem{remark}[thm]{Remark}

\theoremstyle{remark}

\numberwithin{equation}{thm}

\def\Min{\operatorname{Min}}

\def\Im{\operatorname{Im}}

\def\Hom{\operatorname{Hom}}

\def\Max{\operatorname{Max}}
\def\Assh{\operatorname{Assh}}
\def\End{\mathrm{End}}

\def\grade{\operatorname{grade}}

\def\mod{\mathrm{mod}}

\def\m{\mathfrak m}
\def\n{\mathfrak n}

\def\H{\mathrm{H}}

\newcommand{\Ann}{\mathrm{Ann}}

\newcommand{\rmi}{\mathrm{i}}

\newcommand{\rmr}{\mathrm{r}}

\newcommand{\rmE}{\mathrm{E}}

\newcommand{\rmH}{\mathrm{H}}

\newcommand{\rmK}{\mathrm{K}}

\newcommand{\rmQ}{\mathrm{Q}}

\newcommand{\rmU}{\mathrm{U}}
\newcommand{\rmV}{\mathrm{V}}
\newcommand{\rmW}{\mathrm{W}}

\newcommand{\calF}{\mathcal{F}}

\newcommand{\calR}{\mathcal{R}}

\newcommand{\fka}{\mathfrak{a}}

\newcommand{\fkc}{\mathfrak{c}}

\newcommand{\fkp}{\mathfrak{p}}
\newcommand{\fkq}{\mathfrak{q}}

\newcommand{\mapright}[1]{%
\smash{\mathop{%
\hbox to 1cm{\rightarrowfill}}\limits^{#1}}}

\newcommand{\mapleft}[1]{%
\smash{\mathop{%
\hbox to 1cm{\leftarrowfill}}\limits_{#1}}}

\def\depth{\operatorname{depth}}
\def\AGL{\operatorname{AGL}}

\def\Ass{\operatorname{Ass}}

\def\Tr{\operatorname{Tr}}
\def\height{\mathrm{ht}}
\def\Spec{\operatorname{Spec}}

\title[When are the rings $I:I$ Gorenstein?]{When are the rings $I:I$ Gorenstein?}

\author[Naoki Endo]{Naoki Endo}
\address{Department of Mathematics, Faculty of Science Division II, Tokyo University of Science, 1-3 Kagurazaka, Shinjuku, Tokyo 162-8601, Japan}
\email{nendo@rs.tus.ac.jp}
\urladdr{https://www.rs.tus.ac.jp/nendo/}

\author[Shiro Goto]{Shiro Goto}
\address{Department of Mathematics, School of Science and Technology, Meiji University, 1-1-1 Higashi-mita, Tama-ku, Kawasaki 214-8571, Japan}
\email{shirogoto@gmail.com}

\author[Shin-ichiro Iai]{Shin-ichiro Iai}
\address{Mathematics laboratory, Sapporo College, Hokkaido University of Education, 1-3 Ainosato 5-3, Kita-ku, Sapporo 002-8502, Japan}
\email{iai@sap.hokkyodai.ac.jp}

\author[Naoyuki Matsuoka]{Naoyuki Matsuoka}
\address{Department of Mathematics, School of Science and Technology, Meiji University, 1-1-1 Higashi-mita, Tama-ku, Kawasaki 214-8571, Japan}
\email{naomatsu@meiji.ac.jp}

\thanks{The first author was partially supported by JSPS Grant-in-Aid for Young Scientists 20K14299. 
The second author was partially supported by JSPS Grant-in-Aid for Scientific Research (C) 21K03211. }

\begin{document}

\maketitle

\setlength{\baselineskip} {15.5pt}

\begin{abstract}
Let $I ~(\ne A)$ be an ideal of a $d$-dimensional Noetherian local ring $A$ with $\height_AI \ge 2$, containing a non-zerodivisor. The problem of when the ring $I:I=\End_AI$ is Gorenstein is studied, in connection with the problem of the Gorensteinness in Rees algebras $\calR_A(Q^d)$ for certain parameter ideals $Q$ of $A$, that was closely explored by the preceding paper \cite{GI} of the authors. Examples are given.
\end{abstract}


\section{Introduction}
Let $(A,\m)$ be a Noetherian local ring with $d=\dim A \ge 2$ and $t=\depth A \ge 1$. For each ideal $I$ of $A$ we set $\calR_A(I) = \bigoplus_{n \ge 0}I^n$ and call it the Rees algebra of $I$. In the previous paper \cite{GI}, the authors are particularly interested in the question of when the Rees algebra $\calR_A(Q^N)$ is a Gorenstein ring, where $Q$ is a parameter ideal of $A$ and $N \ge 1$. This question is already settled, when $A$ is a Cohen-Macaulay ring, and it is known that $N=d$ or $d-1$ (\cite[Theorem (1.2)]{GS}, \cite[Lemma 2.4]{Hoa}). Nevertheless, even though $A$ is not a Cohen-Macaulay ring and $N =d \ge 2$, the Rees algebra $\calR_A(Q^d)$ can be a Gorenstein ring. As far as we know, the first result was reported by Y. Shimoda around 30 years ago in his seminar talk at Meiji University, who showed that provided $d=2$ and $t=1$, $\calR_A((a,b)^2)$ is a Gorenstein ring, under certain specific conditions on the system $a,b$ of parameters of $A$. His argument was recently rediscovered and motivated the researches \cite{GI}. Together with many other results, the authors succeeded in providing a complete generalization of Shimoda's theorem in the following way.

\begin{thm}[{\cite[Theorem 1.1]{GI}}]\label{main0}
Assume that $\H^i_\m(A)=(0)$ for all $i \not\in \{1,d\}$ and $\H^1_\m(A)$ is a finitely generated $A$-module. Let $Q=(a_1, a_2, \ldots, a_d)$ be a parameter ideal of $A$. Then the following conditions are equivalent.
\begin{enumerate}[$(1)$]
	\item $\calR_A(Q^d)$ is a Gorenstein ring.
	\item $\rmH_\m^1(A) \ne (0)$, $\rmr_A(\rmH_\m^1(A))=1$, and $(0):_A\H_\m^1(A)= \sum_{i=1}^d\rmU(a_iA)$.
\end{enumerate}
When this is the case, the $(S_2)$-ification $\widetilde{A}$ of $A$ is a Gorenstein ring.
\end{thm}
\noindent
Here, $\rmH_\m^i(*)$ stands for the local cohomology functor of $A$ with respect to $\m$, $\rmr_A(M)$ denotes, for each Cohen-Macaulay $A$-module $M$, the Cohen-Macaulay type of $M$, and $\rmU(a_iA)$ denotes the unmixed component of the ideal $a_iA$ in $A$. Added to it, they showed that once $\calR_A(Q^d)$ is a Gorenstein ring for some parameter ideal $Q$ of $A$ and the $(S_2)$-ification $\widetilde{A}$ is a Cohen-Macaulay ring, the basic hypothesis that $\H^i_\m(A)=(0)$ for all $i \not\in \{1,d\}$ and the $A$-module $\H^1_\m(A)$ is finitely generated should be naturally satisfied.

The theorems in \cite{GI}, especially the above theorem \ref{main0}, give a clear characterization for $\calR_A(Q^d)$ to be a Gorenstein ring, establishing a generalization of Shimoda's result. Nevertheless, it should be noticed here that it seems not so easy to provide ample examples of $A$ and $Q$, which satisfy Condition (2) of Theorem \ref{main0}. Because the $(S_2)$-ification $\widetilde{A}$ of $A$ is a Gorenstein ring once $\calR_A(Q^d)$ is Gorenstein, there might be a different approach towards the construction of new examples, based on the Gorensteinness of $(S_2)$-ifications, and if we can do this, it could provide a further viewpoint,  not only for the theory of Gorenstein Rees algebras, but also for the study of $(S_2)$-ifications. This expectation has strongly motivated the present researches.

Let us now state our own results, explaining how this paper is organized. To develop our arguments, we need some basic results on $(S_2)$-ifications and those on trace ideals as well, which we shall briefly summarize in Section 2. In Section 3, we shall discuss the problem of when the rings $I:I$ are Gorenstein rings, where $I~(\ne A)$ is an ideal of $A$ with $\height_AI \ge 2$, containing a non-zerodivisor of $A$. We always consider the colon $I:I$ inside the total ring $\rmQ(A)$ of fractions of $A$, whence $I:I \cong \End_AI$ as an $A$-algebra. The $A$-algebra $I:I$ is closely related to the $(S_2)$-ification $\widetilde{A}$ of $A$, which we will explain in Section 3, and eventually we shall give ample concrete examples of the local rings $A$, which possess Gorenstein Rees algebras $\calR_A(Q^d)$ for some parameter ideals $Q$.

Section 4 is devoted to more constructions of the ideals $I$, for which $I:I =\widetilde{A}$ and $I:I$ is a regular ring. We will leave in this section the problem on the Gorensteinness in Rees algebras. Instead, we are interested in finding what values the number $t = \depth A$ can take, when $I:I$ is a Gorenstein ring. In general, we have $0 < t < d$, and for $t=1,2$ we can construct the rings $A$ and ideals $I ~(\ne A)$ such that $\depth A =t$, $\height_AI \ge 2$, and $I:I$ is a regular ring. Although we can show the case where $t = 3$ is also possible, we are not able to cover all the cases of $0 < t < d$. We would like to leave the remainder cases to the interested readers.


\section{Preliminaries}

\subsection{Some basic results on $(S_2)$-ifications}

Let $R$ be an arbitrary commutative ring and let $\rmQ(R)$ denote the total ring of fractions of $R$. We set \begin{itemize}
\item $\operatorname{Ht}_{\ge 2}(R)$=$\{I \mid I~\text{is~an~ideal~of}~R, \height_R I \ge 2 \}$ and 
\item $\rmW(R)=\{a \in R  \mid a~\text{is~a~non-zerodivisor~of}~R \}$. 
\end{itemize}
Throughout, let us fix a $\rmQ(R)$-module $V$ and an $R$-submodule $M$ of $V$. 
\begin{definition}\label{1.1a} $\widetilde{M} = \{f \in V \mid If \subseteq M \ \text{for~some}~I \in \operatorname{Ht}_{\ge 2}(R) \}$. 
\end{definition}
\noindent
Hence, if $L$ is an $R$-submodule of $V$ and $M \subseteq L$, then $\widetilde{M} \subseteq \widetilde{L}$. We call $\widetilde{M}$ the $(S_2)$-ification of $M$. In fact, $\widetilde{M}$ is an $R$-submodule of $V$ containing $M$. In particular, $\widetilde{R}$ considered inside $\rmQ(R)$ is an intermediate ring $R \subseteq \widetilde{R} \subseteq \rmQ(R)$. Notice that 
$$
\widetilde{R}=\{f \in \rmQ(R) \mid \height_R[R:_Rf] \ge 2\}
$$
and that $\widetilde{M}$ is an $\widetilde{R}$-submodule of $V$ also. For $a \in \rmW(R)$, $x \in V$, and an $R$-submodule of $V$, let us denote $\frac{x}{a}=a^{-1}x$ and $\frac{X}{a} = a^{-1}X$ in $V$. We then have, for all $a,b \in \rmW(R)$, that $$\frac{aM : b}{a}= \frac{M}{a} \cap \frac{M}{b}.$$

\begin{lem}\label{1.1}
Suppose that $R$ is a Noetherian ring, $\fkp \in \Spec R$, and $a \in \rmW(R)$. If $a \in \fkp$ and $\height_R\fkp \ge 2$, then $\height_R (a,b) \ge 2$ for some $b \in \rmW(R) \cap \fkp$.
\end{lem}

\begin{proof}
Notice that $\fkp \not\subseteq \bigcup_{P \in \Min_RR/aR}P \cup \bigcup_{P \in \Ass R}P$.
\end{proof}

Let $a,b \in R$ and $N$ an $R$-module. We say that the sequence $a,b$ is $N$-regular, if $a$ is $N$-nonzerodivisor and $b$ is $N/aN$-nonzerodivisor. So, we don't require that $N/(a,b)N \ne (0)$.

\begin{lem}\label{1.1.1}
Let $a,b \in \rmW(R)$. If $\height_R(a,b) \ge 2$, then the sequence $a,b$ is $\widetilde{M}$-regular.
\end{lem}

\begin{proof}
Let $f \in \widetilde{M}$ and assume that $bf=ag$ for some $g \in \widetilde{M}$. We set $x = \frac{f}{a}=\frac{g}{b}$, and choose $I,J \in \operatorname{Ht}_{\ge 2}(R)$ so that $If + Jg \subseteq M$. Then, since $Iax+ Jbx \subseteq M$, we get $(Ia + Jb)x \subseteq M$, whence $x \in \widetilde{M}$ because $Ia + Jb \in \operatorname{Ht}_{\ge 2}(R)$.
\end{proof}

\begin{prop}\label{1.2}
Let $R$ be a Noetherian ring and suppose that one of the following conditions is satisfied.
\begin{enumerate}
\item[$(1)$] $\rmQ(R)M=V$.
\item[$(2)$] $\height_R\fkp \le 1$ for every $\fkp \in \Ass R$.
\end{enumerate}
Then $M=\widetilde{M}$ if and only if every pair $a,b\in \rmW(R)$ with $\height_R(a,b) \ge 2$ is $M$-regular.
\end{prop}

\begin{proof}
It suffices to prove the {\em if} part. Assume that $M \ne \widetilde{M}$ and consider the exact sequence 
$$(\rmE) \ \ \ \ \ \ 0 \to M \to \widetilde{M} \to Z \to 0$$
of $R$-modules, where $Z = \widetilde{M}/M$. Let $\fkp \in \Ass_RZ$. Hence, $\fkp = M:_Rf$ for some $f \in \widetilde{M}$. Choose $I \in \operatorname{Ht}_{\ge 2}(R)$ so that $If \subseteq M$. We then have $I \subseteq \fkp$. Notice that, if Condition (1) is satisfied, $af \in M$ for some $a \in \rmW(R)$, and if Condition (2) is satisfied, $I \not\subseteq \bigcup_{P \in \Ass R}P$. In any case, we get $af \in M$ for some $a \in \rmW(R)$. Therefore, by Lemma \ref{1.1}, $\height_R(a,b) \ge 2$ for some $b \in \rmW(R) \cap \fkp$, whence by Lemma \ref{1.1.1},  the sequence $a,b$ is $M$-regular. Therefore, by Sequence (E), we get the long exact sequence
$$
0 \to (0):_Za \overset{\sigma}{\to} M/aM \to \widetilde{M}/a\widetilde{M} \to Z/aZ \to 0
$$
where $$b\sigma(f~\mod~M) =\sigma(bf~\mod~M)= 0$$ because $bf \in M$. Therefore, $\sigma(f~\mod~M)=0$ since $b$ is a non-zerodivisor for $M/aM$, whence $f \in M$. This is impossible. Thus $M = \widetilde{M}$.
\end{proof}

\begin{cor}\label{1.3}
If $R$ is a Noetherian ring and satisfies the condition $(S_2)$ of Serre, then $R =\widetilde{R}$. 
\end{cor}

\begin{cor}\label{1.4}
With the same assumption as is in Proposition $\ref{1.2}$, the following assertions hold true.
\begin{enumerate}[$(1)$]
\item $\widetilde{\widetilde{M}}= \widetilde{M}$.
\item Let $M \subseteq L \subseteq V$ be an $R$-submodule of $V$. If every pair $a, b \in \rmW(R)$ with $\height_R(a,b) \ge 2$ is $L$-regular, then $\widetilde{M} \subseteq \widetilde{L}=L$.
\end{enumerate}
\end{cor}

In what follows, unless otherwise specified, let $R$ be a Noetherian ring. We assume that the following additional conditions are also satisfied.
\begin{enumerate}
\item $M$ is a finitely generated $R$-module.
\item $\rmQ(R)M=V$.
\item $(0):_{\rmQ(R)}V=(0)$.
\end{enumerate}
Hence, if $a \in \rmW(R)$ and $M/aM \ne (0)$, then $\height_R\fkp=1$ for every $\fkp \in \Min_RM/aM$
 since $(0):_RM=(0)$. Thanks to Condition (2), every $f \in V$ has an expression of the form $f = \frac{m}{a}$ with $a \in \rmW(R)$ and $m \in M$.

Let $a \in \rmW(R)$ and let $$aM = \bigcap_{\fkp \in \Ass_RM/aM}Q(\fkp)$$ be a primary decomposition of $aM$ in $M$. We set
$$\rmU(aM)=\begin{cases}
M, \ \ \ \ \ \ \ \ \ \ \ \ \ \ \ \ \ \ \  \ \ \ \ \text{if}\ \ aM=M.\\
\bigcap_{\fkp \in \Min_RM/aM}Q(\fkp), \ \ \text{if}\ \ aM \ne M.\\
\end{cases}
$$

\begin{thm}\label{1.5}
Let $a \in \rmW(R)$ and $m \in M$. Then $\frac{m}{a} \in \widetilde{M}$ if and only if $m \in \rmU(aM)$.
\end{thm}

\begin{proof}
We may assume $aM \ne M$. Suppose $\frac{m}{a}\in \widetilde{M}$ and choose $I \in \operatorname{Ht}_{\ge 2}$ so that $I \subseteq aM :_Rm$. Let $\fkp \in \Min_RM/aM$. Then, since $\height_R\fkp = 1$, $aM_R:m \not\subseteq \fkp$, so that $m \in [aM]_\fkp \cap M=Q(\fkp)$. Hence, $m \in \rmU(aM)$.

Conversely, suppose $m \in \rmU(aM)$. If $aM = \rmU(aM)$, then $m \in aM$, so that $\frac{m}{a} \in M \subseteq \widetilde{M}$. Suppose $aM \ne \rmU(aM)$ and set $\calF= \Ass_RM/aM \setminus \Min_RM/aM$. Then, $\calF \ne \emptyset$ and for each $\fkp \in \calF$, there is an integer $\ell=\ell(\fkp) \gg 0$ such that $\fkp^\ell M \subseteq Q(\fkp)$, whence $\fkp^\ell (aM) \subseteq Q(\fkp)$. Therefore, setting $\fka = \prod_{\fkp \in \calF}\fkp^{\ell(\fkp)}$, we have $$\fka \rmU(aM) \subseteq \bigcap_{\fkp \in \calF}Q(\fkp) \cap \rmU(aM) =aM.$$ Hence, $\frac{\rmU(aM)}{a} \subseteq \widetilde{M}$, because $\fka \in \operatorname{Ht}_{\ge 2}(R)$.
\end{proof}

\begin{cor}\label{1.6}
If $\widetilde{M} \subseteq \frac{M}{a}$ for some $a \in \rmW(R)$, then $\widetilde{M} = \frac{\rmU(aM)}{a}$. Consequently $$\widetilde{M}= \bigcup_{a \in \rmW(R)}\frac{\rmU(aM)}{a}.$$
\end{cor}

\begin{proof}
Let us check the second assertion. Let $f \in \widetilde{M}$ and write $f = \frac{m}{a}$ with $a \in \rmW(R)$ and $m \in M$. We take $I \in \operatorname{Ht}_{\ge 2}(R)$ so that $Im \subseteq aM$. If $aM=M$, then $f \in \frac{\rmU(aM)}{a}$. Suppose $aM \ne M$. We then have $I \not\subseteq \fkp$ for any $\fkp \in \Min_RM/aM$, so that $m \in [aM]_\fkp \cap M = Q(\fkp)$ for every $\fkp \in \Min_RM/aM$. Thus, $f=\frac{m}{a} \in \frac{\rmU(aM)}{a}$.
\end{proof}

\begin{cor}\label{1.7}
$\widetilde{M}$ is a finitely generated $R$-module if and only if $a\widetilde{M}=\rmU(aM)$ for some $a \in \rmW(R)$.
\end{cor}

\subsection{Trace ideals}
Let $R$ be an arbitrary commutative ring and let $M, X$ be $R$-modules. Let $$\tau: \Hom_R(M,X) \otimes_RM \to X$$ denote the homomorphism defined by $\tau(f \otimes m)=f(m)$ for each $f \in \Hom_R(M,X)$ and $m \in M$. We set $\operatorname{Tr}_X(M)= \Im \tau$ and call it the trace module of $M$ in $X$. In fact, $\operatorname{Tr}_X(M)$ is an $R$-submodule of $X$, and a given $R$-submodule $Y$ of $X$ is called a {\em trace~submodule} in $X$, if $Y = \operatorname{Tr}_X(M)$ for some $R$-module $M$.

The following result is due to H. Lindo \cite{Lindo}.

\begin{prop}[{\cite[Lemma 2.3]{Lindo}}]\label{1.10}
Let $I$ be an ideal of $R$. Then the following conditions are equivalent.
\begin{enumerate}[$(1)$]
\item $I$ is a trace ideal in $R$, that is $I = \Tr_R(M)$ for some $R$-module $M$.
\item $I = \Tr_R(I)$.
\item For each homomorphism $f : I \to R$ of $R$-modules, there is an endomorphism $g : I \to I$ such that $f = \iota{\cdot}g$, where $\iota : I \to R$ denotes the embedding.
\end{enumerate}
When $I$ contains a non-zerodivisor, one can add the following, where the colons are considered inside the total ring $\rmQ(R)$ of fractions of $R$.

\begin{enumerate}[$(1)$]
\item[$(4)$] $I:I = R:I$. 
\end{enumerate}
\end{prop}

\section{When are the rings $I:I$ Gorenstein?}\label{main}

We are now in a position to discuss the question of when the $(S_2)$-ifications are Gorenstein rings. Let us begin with the following.

\begin{lem}\label{lemma2}
Let $A$ be a Noetherian ring. Let $A \subseteq B \subseteq \rmQ(A)$ be a subring of $\rmQ(A)$ and assume that $B$ is a finitely generated $A$-module. Let $I$ be an ideal of $B$ such that $I \subseteq A$ and $\height_AI \ge 2$. If  $A$ is locally quasi-unmixed and $B$ satisfies $(S_2)$, then $I$ is a trace ideal in $A$ and $B=I:I$. 
\end{lem}

\begin{proof}
Since $A$ is locally quasi-unmixed, we have $\height_BP=\height_A(P \cap A)$ for every $P \in \Spec B$ (\cite[THEOREM 3.8]{Ratliff}), so that $\height_BI \ge 2$, whence $\grade_BI \ge 2$ because $B$ satisfies $(S_2)$. Therefore, $B = I:I$, so that 
$$A:I \subseteq B:I=B =I:I \subseteq A:I,$$ which implies, by Proposition \ref{1.10}, that $I$ is a trace ideal in $A$.
\end{proof}

\begin{ex}
Let $k$ be a field and $S=k[[X,Y,Z,W]]$ be the formal power series ring over $k$. We consider the ring $A = S/[(X,Y) \cap (Z,W)]$ and let $\m$ denote the maximal ideal of $A$. Then, $\m = A:B$, where $B =S/(X,Y) \times S/(Z,W)$, and Lemma \ref{lemma2} tells us that for every $\ell \ge 1$, $\m^\ell$ is a trace ideal in $A$ and $B = \m^\ell : \m^\ell$. 
\end{ex}

\begin{prop}\label{prop1.1}
Let $A$ be a Noetherian local ring and $I~(\ne A)$ an ideal of $A$ with $\height_AI \ge 2$. Assume that there exists an exact sequence 
$$0 \to A \to \rmK_A \to C \to 0$$ of $A$-modules such that $IC=(0)$.
Then the following assertions hold true, where $R=\widetilde{A}$.
\begin{enumerate}[$(1)$]
\item $R \cong \rmK_A$ as an $A$-module.
\item $\Hom_A(R,\rmK_A) \cong R$ as an $R$-module. 
\item If $\rmK_A$ is a Cohen-Macaulay $A$-module, then $R$ is a Gorenstein ring.
\item If $I$ is a trace ideal in $A$, then $R=I:I$.
\end{enumerate}
\end{prop}

\begin{proof}
(1), (2) Let $\fkp \in \Ass A$. Then $I \not\subseteq \fkp$ because $\fkp \in \Ass_A\rmK_A = \Assh A$, so that $C_\fkp=(0)$ and 
$$
A_\fkp \cong (\rmK_A)_\fkp =\rmK_{A_\fkp}
$$
(\cite[Theorem 4.2]{A}, see also \cite[Satz 5.22]{HK}). Hence, $A_\fkp$ is a Gorenstein ring for every $\fkp \in \Ass A$, that is $\rmQ(A)$ is a Gorenstein ring. We choose an $A$-submodule $K$ of $\rmQ(A)$ so that $K \cong \rmK_A$ as an $A$-module. Hence, $\rmQ(A)K=\rmQ(A)$, because $\rmQ(A)$ is self-injective. Then, since $I$ contains a subsystem of parameters of $A$ of length $2$, we get, taking the $K$-dual of the sequence $0 \to A \overset{\varphi}{\rightarrow} K \to C \to 0$, the natural isomorphism 
$$
\psi : K:K = \Hom_A(K,K) \overset{\varphi^*}{\rightarrow} \Hom_A(A,K) = K
$$
of $A$-modules, where $\psi(1)= \varphi(1)$. Hence, $R \cong \rmK_A$ as an $A$-module since $R= K:K$ by \cite[Theorem 1.6]{AG}, which shows Assertion (1). On the other hand, since $\psi(1)= \varphi(1)$, $(K:K)/A \cong C$ as an $A$-module, so that $$IR=I(K:K) \subseteq A.$$ We now notice that $R \cong K$ as an $R$-module, because $R \cong K$ as an $A$-module and $K$ is an $R$-submodule of $\rmQ(A)$ such that $\rmQ(A)K=\rmQ(A)$. Let $K = \alpha R$ for some $\alpha$ of $\rmQ(A)$. Then, $\alpha$ is invertible in $\rmQ(A)$, so that 
$$ 
K:R= \alpha R : R = \alpha[R:R]=\alpha R=K,$$
which shows Assertion (2).

(3) Since $\depth_AR=\dim A$ by Assertion (1), every system of parameters of $A$ forms a regular sequence for $R$, whence $\height_RM= \dim A$ for every $M \in \Max R$, so that $R_M$ is a Gorenstein ring by Assertion (2) (\cite[Korollar 5.14]{HK}).

(4) Suppose that $I$ is a trace ideal in $A$ and set $B = I:I$. Then, since $IR \subseteq A$, we have $R \subseteq A : I = B$ where the equality follows from Proposition \ref{1.10}, while $B \subseteq \widetilde{A} = R$ by Definition \ref{1.1a}, since $IB=I \subseteq A$ and $\height_AI \ge 2$. Consequently, $R=B$, whence $\widetilde{A}=I:I$. 
\end{proof}

\begin{cor}\label{prop2}
Let $A$ be a Noetherian local ring with $\dim A=d$ and $t = \depth A$. Let $A \subseteq B \subseteq \rmQ(A)$ be a subring of $\rmQ(A)$ such that $B$ is a finitely generated $A$-module. We set $\fka = A:B$ and assume the following three conditions are satisfied.
\begin{enumerate}
\item[$(\rm{1})$] $A$ is a quasi-unmixed ring.
\item[$(\rm{2})$] $\height_A \fka \ge 2$.
\item[$(\rm{3})$] $B$ is a Gorenstein ring.
\end{enumerate} 
Then the following assertions hold true.
\begin{enumerate}[$(\rm{a})$]
\item[$(\rm{a})$] $B=\widetilde{A}$, $\depth_AB=d$, $\fka$ is a trace ideal in $A$, and $B = \fka : \fka$.
\item[$(\rm{b})$] $\Ass A = \Assh A$.
\item[$(\rm{c})$] $A$ possesses the canonical module $\rmK_A$, and $\rmK_A \cong B$ as an $A$-module.
\item[$(\rm{d})$] $A=B$ if and only if $t=d$.
\end{enumerate} 
\end{cor}

\begin{proof}
We have,  by Lemma  \ref{lemma2}, $B = \fka : \fka$ and $\fka$ is a trace ideal in $A$. Because $\height_BM= \height_A\m=d$ for all $M \in \Max B$, every system of parameters of $A$ forms a regular sequence in $B_M$, so that it forms a regular sequence for the $A$-module $B$. Hence, $\depth_AB = d$. Let $C = B/A$. Then $\dim_AC \le d-2$ since $\fka C=(0)$, so that $\operatorname{H}_{\m}^d(A) \cong \operatorname{H}_{\m}^d(B)$ as an $A$-module (here $\m$ denotes the maximal ideal of $A$). Therefore, $\rmK_{\widehat{A}}\cong \widehat{A} \otimes_A  B$ as an $\widehat{A}$-module where $\widehat{A}$ denotes the $\m$-adic completion of $A$, whence $A$ possesses the canonical module $\rmK_A$ and $\rmK_A \cong B$ as an $A$-module. We have $B \subseteq \widetilde{A}$ since $\height_A\fka \ge 2$, while $\widetilde{A} \subseteq \widetilde{B} =B$ by Corollary \ref{1.4} (2). Hence, $B = \widetilde{A}$. 
Notice that $\Ass A \subseteq \Ass_AB = \Ass_A\rmK_A = \Assh A$, and we have $\Ass A = \Assh A$. If $A$ is a Cohen-Macaulay ring, then the depth lemma tells us that $\depth_AC \ge d-1$, which forces $C=(0)$ because $\dim_AC \le d-2$. Hence, $A=B$ if $t=d$, which completes the proof. 
\end{proof}

\begin{thm}\label{1}
Let $A$ be a Noetherian local ring with $d = \dim A \ge 2$ and $\depth A \ge 1$. Assume that $A$ is a quasi-unmixed ring. Let $I~(\ne A)$ be an ideal of $A$  with $\height_A I\ge 2$ and assume that $I$ contains a non-zerodivisor of $A$. We set $B = I:I$. Then the following conditions are equivalent.
\begin{enumerate}[$(1)$]
\item $B$ is a Gorenstein ring.
\item $A$ possesses the canonical module $\rmK_A$ and $I$ is a trace ideal in $A$ such that $(\rmi)$ $\rmK_A$ is a Cohen-Macaulay $A$-module and $(\rm{ii})$ there exists an exact sequence
$$0 \to A \to \rmK_A \to C \to 0$$
of $A$-modules such that $IC=(0)$.
\item $\depth_AB=d$, $A$ possesses the canonical module $\rmK_A$, and $B \cong \rmK_A$ as an $A$-module.
\end{enumerate}
When this is the case, $A$ is an unmixed ring with $B=\widetilde{A}$, and $\fka=A:B$ is a trace ideal in $A$ with $B=\fka : \fka$.
\end{thm}

\begin{proof}

$(1) \Rightarrow (3)$~See Corollary \ref{prop2}.

$(3) \Rightarrow (2)$ Notice that $\grade_BI = \height_BI \ge 2$, since $B$ is a Cohen-Macaulay ring. We have the exact sequence $$0 \to A \to \rmK_A \to C \to 0$$ of $A$-modules such that $IC=(0)$, while by Lemma \ref{lemma2} $I$ is a trace ideal in $A$.

$(2) \Rightarrow (1)$ Since $\depth_A\rmK_A= d$, by Proposition \ref{prop1.1} $B$ is a Gorenstein ring.

See Lemma \ref{lemma2} and Corollary \ref{prop2} for the last assertion.
\end{proof}

\begin{cor}\label{2}
Let $A$ be a Noetherian local ring and let $I~(\ne A)$ be an ideal of $A$  with $\height_A I \ge 2$. Assume that $I$ contains a non-zerodivisor of $A$. We set $B = I:I$. Then the following conditions are equivalent.
\begin{enumerate}[$(1)$]
\item $B$ is a Gorenstein ring, $A$ is a homomorphic image of a Cohen-Macaulay ring, and $\Min A= \Assh A$.
\item $A$ possesses the canonical module $\rmK_A$ and $I$ is a trace ideal in $A$ such that $(\rmi)$ $\rmK_A$ is a Cohen-Macaulay $A$-module and $(\rm{ii})$ there exists an exact sequence
$$0 \to A \to \rmK_A \to C \to 0$$
of $A$-modules such that $IC=(0)$.
\item $\depth_AB=d$, $A$ possesses the canonical module $\rmK_A$, and $B \cong \rmK_A$ as an $A$-module.
\end{enumerate}
When this is the case, $A$ is an unmixed ring with $B=\widetilde{A}$, and $\fka=A:B$ is a trace ideal in $A$ with $B=\fka : \fka$.
\end{cor}

\begin{proof}
Suppose that Condition (2) or (3) is satisfied. Then, $(0):_A\rmK_A= (0)$ and $\rmK_A$ is a Cohen-Macaulay $A$-module, so that all the formal fibers of $A$ are Cohen-Macaulay, while $A$ is unmixed since $A$ is a submodule of $\rmK_A$. Therefore, thanks to Kawasaki's theorem \cite[Theorem 1.1]{kawasaki}, $A$ is a homomorphic image of a Cohen-Macaulay ring. On the other hand, once $A$ is a homomorphic image of a Cohen-Macaulay ring with $\Min A = \Assh A$, $A$ is quasi-unmixed. These observations enable us to assume, from the beginning, that $A$ is quasi-unmixed, whence the assertion follows from Theorem \ref{1}.
\end{proof}

\section{Relationship with Gorenstein Rees algebras of powers of parameters}\label{powers}

\begin{thm}[{cf. \cite{GI}}]\label{GI}
Let $(A,\m)$ be a Noetherian complete local ring such that $d=\dim A \ge 2$, $\depth A=1$, and $\Min A = \Assh A$. Let $I$ be an $\m$-primary ideal of $A$. We set $B = I :I$ and $\fka = A:B$, and suppose that the following three conditions are satisfied.
\begin{enumerate}
\item[$(1)$] $B$ is a Gorenstein ring.
\item[$(2)$] $A \ne B$ and $\rmr_A(B/A)=1$, that is the socle of the $A$-module $B/A$ has length one.
\item[$(3)$] $\fka =(a_1, a_2, \ldots, a_d)B$ for some $a_1, a_2, \ldots, a_d \in \m$.
\end{enumerate}
Then, $B = \fka : \fka$, and  the Rees algebra $\calR_A(Q^d)$ of $Q^d$ is a Gorenstein ring, where $Q=(a_1, a_2, \ldots, a_d)$ in $A$.
\end{thm}

\begin{proof}
We have $B=\fka : \fka$ by Theorem \ref{1}. Since $I{\cdot}(B/A)=(0)$, applying the functor $\rmH_\m^i(*)$ to the exact sequence $$0 \to A \to B \to B/A \to 0$$
of $A$-modules, we get $\rmH_\m^1(A) \cong B/A$ and $\rmH_\m^i(A) = (0)$ for all $i \not\in \{1, d\}$. Hence $$(0):_A\rmH_\m^1(A) = \fka\ \ \ \text{and} \ \  \rmr_A(\rmH^1_\m(A))=1.$$ On the other hand, since $\depth_AB=d$ and $a_1, a_2, \ldots, a_d$ forms a system of parameters in $A$, the sequence $a_1, a_2, \ldots, a_d$ is $B$-regular. Hence, each $a_i$ is a non-zerodivisor of $A$ and $a_i B \subseteq A$, so that $B = \widetilde{A} \subseteq a_i^{-1}A$, whence $B =a_i^{-1}\rmU(a_iA)$ by Corollary \ref{1.6}, where $\rmU(a_iA)$ denotes the unmixed component of the ideal $a_iA$. Hence $$\fka = \sum_{i=1}^da_iB= \sum_{i=1}^d \rmU(a_iA).$$
Consequently, $\calR_A(Q^d)$ is a Gorenstein ring, thanks to Theorem \ref{main0}.  
\end{proof}

\subsection{The simplest examples given by specific Buchsbaum rings}
Let $(S,\n)$ be a Gorenstein complete local ring with $d = \dim S \ge 2$ and assume that $S$ contains a coefficient field $k$. Let $\fkq=(a_1, a_2, \ldots, a_d)S$ be a parameter ideal of $S$ such that $\fkq \ne \n$. We set $A = k +\fkq$. Then, $A$ is a subring of $S$, and $\fkq$ is a maximal ideal in $A$, since $k \cong A/\fkq$. We have $\ell_A(S/A) = \ell_A(S/\fkq) -1 < \infty$, where $\ell_A(*)$ denotes the length. Therefore, $S$ is a finitely generated $A$-module, so that $(A, \fkq)$ is a Noetherian complete local ring with $\dim A = d$. Let $\m~(=\fkq)$ stand for the maximal ideal of $A$. We then have $\rmH_\m^i(A)=(0)$ for all $i \not\in \{1,d\}$ and $\rmH_\m^1(A) = S/A$, because $\depth_AS=d$ and $\ell_A(S/A) < \infty$. Hence, $\depth A = 1$, and $A$ is a Buchsbaum ring (\cite{G1, G2}). We have $S = \m : \m$, and $\Ass A = \Ass_AS = \Assh A$, since $\depth_AS=d$. Thus, we get the following.

\begin{thm}
Suppose that $\ell_S(S/\fkq)=2$. Then, $\calR_A(Q^d)$ is a Gorenstein ring, where $Q = (a_1, a_2, \ldots, a_d)A$.
\end{thm}

\begin{cor}
Let $S= k[[X_1, X_2, \ldots, X_d]]$~$(d \ge 2)$ be the formal power series ring over a field $k$ and let $\fkq = (X_1^2, X_2, \ldots, X_d)S$. Then, $\calR_A(Q^d)$ is a Gorenstein ring, where $A = k + \fkq$ and $Q=(X_1^2, X_2, \ldots, X_d)A$.
\end{cor}

\subsection{The case where $A$ is an integral domain}
Let $B=k[[t_1, t_2, \ldots, t_n,s]]$ ($n \ge 1$) be the formal power series ring over a field $k$, and set $V = k[[t_1, t_2, \ldots, t_n]]$, $d=n+1$. Hence, $B =V[[s]]$. We choose a subring $P$ of $V$ so that $V$ is a finitely generated $P$-module and $P:V \ne (0)$. Therefore,  $P$ is a Noetherian local ring and $V=\overline{P}$. We set $\fkc = P:V$ and assume that $\sqrt{\fkc}=\n$, where $\n$ denotes the maximal ideal of $P$.

We set $A = P+sB$. Hence, $A$ is a subring of $B$ and $B$ is a finitely generated $A$-module, because $B/A$ is a homomorphic image of $B/sB$ and $B/sB$ is a finitely generated $P$-module. Therefore, $A$ is a Noetherian complete local ring with $\dim A=d \ge 2$, and $B = \overline{A}$. We set $\fka = A:B$. Let $\m$ denote the maximal ideal of $A$.

\begin{lem}\label{prop1}
$\fka=\fkc + sB$ and $\fka$ is an $\m$-primary ideal of $A$. Hence, $B=\fka:\fka$ and $\depth A=1$.
\end{lem}

\begin{proof}
Let $b \in B$ and write $b = v + x$ with $v \in V$ and $x \in sB$. We then have, for each $c \in \fkc$, $cb = cv+cx \in P + sB =A$, whence $\fkc+sB \subseteq \fka$. Conversely, let $\varphi \in \fka$ and write $\varphi = a + x$ with $a \in P$ and $x \in sB$. Then $a \in A:B$, so that $aV \subseteq A \cap V= P + (sB \cap V) =P$, whence $a \in \fkc $ and  $\varphi \in \fkc + sB$. Therefore, $\fka = \fkc + sB$, so that $\fka$ is an $\m$-primary ideal of $A$. Consequently, $B=\fka : \fka$ by Lemma \ref{lemma2}. The last assertion follows from the fact that $\depth_AB=d$ and $\ell_A(B/A) < \infty$.
\end{proof}

\begin{prop}
$\rmr_A(B/A) = \rmr_P(V/P)$.
\end{prop}

\begin{proof}
Notice that $V/P \cong B/A$ as a $P$-module, since $V=B/sB$ and $P=A/sB$. We then have $$\rmr_A(B/A) = \rmr_{A/sB}(B/A)=\rmr_P(B/A) =\rmr_P(V/P)$$
as claimed.
\end{proof}

\begin{cor}\label{3} If $\fkc = (a_1, a_2, \ldots, a_n)V$ for some $a_1, a_2, \ldots, a_n \in P$ and $\rmr_P(V/P)=1$, then $\calR_A(Q^d)$ is a Gorenstein ring, where $Q=(a_1, a_2, \ldots, a_n, s)$.
\end{cor}

\begin{proof}
We get $\fka = (a_1, a_2, \ldots, a_n, s)B$ by Lemma \ref{prop1}. Hence, the assertion follows from Theorem \ref{GI}.
\end{proof}

\begin{thm}
For each one of the following cases, we have $B= \fka : \fka$ and the assumptions of Corollary $\ref{3}$ are satisfied. Therefore, $\calR_A(Q^d)$ is a Gorenstein ring for the appropriate parameter ideal $Q$ of $A$.
\begin{enumerate}[$(1)$]
\item $n=1$ and $P = k[[H]]$, where $H$ is a symmetric numerical semigroup such that $1 \not\in H$.
\item $n=1$ and $P=k[[t^2+ t^3, t^4, t^6]]$, where $t=t_1$. 
\item Let $k/k_0$ be an extension of fields with $[k:k_0]=2$. Choose $\alpha \in k \setminus k_0$ and set $P=k_0[[t_1, t_2, \ldots, t_n, \alpha t_1, \alpha t_2, \ldots, \alpha t_n]]$ in $V=k[[t_1, t_2, \ldots, t_n]]$.
\end{enumerate}
Notice that for Case $(2)$, $P$ is not the semigroup ring for any numerical semigroup.
\end{thm}

\begin{proof}
See Lemma \ref{prop1} for the equality $B =\fka : \fka$.

(1)~We have $\fkc = t^c V$ $(t=t_1)$ for some $c >0$, while $\rmr_P(V/P)=1$, because $P$ is a Gorenstein ring and $V/P$ is a $P$-submodule of $\rmH_\n^1(P)$.

(2)~Since $\n V=t^2V$, we have $V=P+tP$. First, suppose $\rm{ch}(k)=2$. Then, $(t^2+t^3)^2= t^4+t^6$, so that $P=k[[t^2+t^3, t^4]]$ and $t^7 \in P$. Therefore, $V=\overline{P}$, and $P$ is a Gorenstein ring. Suppose that $\rm{ch}(k) \ne 2$. We then have $t^5 \in P$ because $(t^2+t^3)^2 \in P$, so that $k[[t^2+t^3, t^5]] \subseteq P$. Hence $V =\overline{P}$. Because $$\left<2,5 \right> \subseteq v(k[[t^2+t^3,t^5]]) \subseteq v(P)$$ and $3 \not\in v(P)$, we have $\left<2,5 \right>=v(k[[t^2+t^3,t^5]])=v(P)$, whence $k[[t^2+t^3,t^5]]=P$. In any case, $P$ is a Gorenstein ring but not the semigroup ring for any numerical semigroup, since $t^3 \not\in P$. We have $\fkc= t^cV$ for some $c \ge 2$, since $t^{c_0}V \subseteq P$ for all $c_0 \gg 0$.

(3)~We have $V = P + \alpha P$ and $\n V = \n$, since $\alpha^2 \in k=k_0+k_0\alpha$. Consequently, $\fkc = \n$ since $P \ne V$, so that $\fkc = (t_1, t_2, \ldots, t_n)V$, while $\rmr_P(V/P)=1$, since $V/P \cong P/\n$ as a $P$-module.
\end{proof}

\subsection{The case where $A$ is a fiber product}

Let $S$ be a Gorenstein complete local ring with $d=\dim S \ge 2$ and let $a_1, a_2, \ldots, a_d$ be a system of parameters of $S$. Let $\fkq=(a_1, a_2, \ldots, a_d)$, $T=S/\fkq$, and $B = S \times S$. We consider the fiber product $A= S \times_{T} S$. By definition $$A = \{(x,y) \in B \mid x \equiv y~\mod~\fkq \}$$
and there is the exact sequence
$$
(\rmE)\ \ \ \ \ \ 0 \to A \overset{\iota}{\to} B \overset{\varphi}{\to} T \to 0
$$
of $A$-modules, where $\iota : A \to B$ denotes the embedding and $\varphi: B \to T,\  \varphi(x,y) = x-y ~\mod~\fkq$. Because $A$ is a subring of $B$ and $B$ is a finitely generated $A$-module, $A$ is a Noetherian complete local ring with $\dim A = d$, while $\depth A = 1$, since $\dim T = 0$. Let $\alpha_i=(a_i, a_i ) \in A$ for each $1 \le i \le d$ and set $Q = (\alpha_1, \alpha_2, \ldots, \alpha_d)$. Then, $Q$ is a parameter ideal of $A$ and we have the following.

\begin{thm}\label{4}
The Rees algebra $\calR_A(Q^d)$ is a Gorenstein ring.
\end{thm}

\begin{proof}
The exact sequence $(\rmE)$ shows $\rmr_A(B/A)=\rmr(T)=1$. The ring $B$ is Gorenstein with $A:B= \Ann_AT= QB$, so that $B = QB:QB$ by Lemma \ref{lemma2}. Hence, the assertion follows from Theorem \ref{GI}
\end{proof}

\begin{cor}
Let $U=k[[X_1, X_2, \ldots, X_d, Y_1, Y_2, \ldots, Y_d]]$ $(d \ge 2)$ be the formal power series ring over a field $k$ and set $A=U/[(X_1, X_2, \ldots, X_d) \cap (Y_1, Y_2, \ldots, Y_d)]$. Let $z_i$ denote, for each $1 \le i \le d$, the image of $X_i + Y_i$ in $A$. Then, $\calR_A(Q^d)$ is a Gorenstein ring, where $Q=(z_1, z_2, \ldots, z_d)$.
\end{cor}

\begin{proof}
Since $A \cong S \times_k S$ where $S=k[[X_1, X_2, \ldots, X_d]]$ and $k = S/(X_1, X_2, \ldots, X_d)$, the result readily follows from Theorem \ref{4}.
\end{proof}

\begin{remark}
More generally, let $t \in S$ and let $A = S \overset{t}{\ltimes} \fkq$ denote the amalgamated duplication of $S$ along $\fkq$ (\cite{AF}; see \cite[Section 3]{EGI} also). Let $Q= (\alpha_1, \alpha_2, \ldots, \alpha_d)$, where $\alpha_i =(a_i, 0) \in A$. Then, for every $t \in S$, $Q$ is a parameter ideal of $A$, and $\calR_A(Q^d)$ is a Gorenstein ring. Let us emphasize that if $t=1$, then $A=S \times_{T} S$, which is the fiber product, and if $t=0$, then $A=S \ltimes \fkq$, which is the idealization of $\fkq$ over $S$.
\end{remark}

\section{More constructions}
In this section we leave the problem on the Gorensteinness in Rees algebras. Instead, we are interested in finding what values the number $t = \depth A$ can take, when $I:I$ is a Gorenstein ring.

First of all, let $R$ be an arbitrary commutative ring, and let $\{I_i\}_{1 \le i \le \ell}$  be a family of ideals of $R$, where $\ell \ge 2$. We set $J_i = I_1 \cap \cdots \overset{\vee}{I_i} \cap \cdots \cap I_\ell$ for each $1 \le i \le \ell$. We then have the following.

\begin{lem}\label{3.1}  $\bigcap_{i=1}^\ell (I_i + J_i) = \sum_{i=1}^\ell J_i$.
\end{lem}

\begin{proof}
This follows, for example, by induction on $\ell$. The detail is left to the reader.
\end{proof}

Let $S = \bigoplus_{i=1}^\ell R/I_i$. Assume that $\bigcap_{i=1}^\ell I_i= (0)$ and set $I = R : S$.

\begin{prop}\label{3.2}
The following assertions hold true.
\begin{enumerate}[$(1)$]
\item $I = \sum_{i=1}^\ell J_i$.
\item $I \subseteq I_i + I_j$, if $i \ne j$.
\item $\rmV(I) = \bigcup_{i \ne j} \rmV(I_i + I_j)$ in $\Spec R$.
\end{enumerate}
\end{prop}

\begin{proof}
(1) We set $\mathbf{e}_i = (0, \cdots, \overset{\overset{i}{\vee}}{1}, \cdots, 0) \in S$ for each $1 \le i \le \ell$. Let $a \in R$ and we have 
\begin{eqnarray*}
a \mathbf{e}_i \in R \ \text{for~all}~1 \le i \le \ell &\Leftrightarrow& ~\text{for~each}~1 \le i \le \ell \  \text{there~exists}~f \in R \ \text{such~that}\\
&{}&(0, \cdots, \overset{i}{\overline{a}}, \cdots, 0)=(\overline{f},\overline{f}, \cdots, \overline{f})~\text{in}~S\\ 
&\Leftrightarrow&~\text{for~each}~1 \le i \le \ell~\text{there~exists}~f \in J_i ~\text{such~that}~a -f \in I_i \\
&\Leftrightarrow& a \in \bigcap_{i=1}^\ell (I_i + J_i)=\sum_{i=1}^\ell J]_i,
\vspace{-2em}
\end{eqnarray*}
where $\overline{c}$ denotes, for $c \in R$ and $1 \le j \le \ell$, the image of $c$ in $R/I_j$. Hence $I = \sum_{i=1}^\ell J_i$.

(2), (3) These are straightforward.
\end{proof}

Let us consider a more specific situation. Let $k$ be a field and let $n, \ell \ge 2$ be integers. Let $T=k[[X_1, X_2, \cdots, X_n]]$ denote the formal power series ring over $k$. Let $F_1, F_2, \ldots, F_\ell$ be subsets of $\{ X_1, X_2, \ldots, X_n \}$ and assume that 
\begin{enumerate}[$(1)$]
\item $F_i \ne \emptyset$ for every $1 \le i \le n$ and 
\item $F_i \not\subseteq F_j$ if $i \ne j$. 
\end{enumerate}
We set $P_i=(F_i) \in \Spec T$ and $\fka = \bigcap_{i=1}^\ell P_i$. Let  $A= T/\fka$, $B = \bigoplus_{i=1}^\ell T/P_i$, and $I = A:B$. Then, $\fka = \bigcap_{i=1}^\ell P_i$ is a reduced primary decomposition of $\fka$ in $T$, and $B =\overline{A}$. Therefore, $\Ass_TA= \{P_1, P_2, \ldots, P_\ell\}$, and $\dim A = \max \{n- |F_i| \mid  1 \le i \le \ell \}$, where $|*|$ denotes the number of elements. Since $\ell \ge 2$, $A \subsetneq B$, and setting $\fkp_i=P_i/\fka$ for each $1 \le i \le n$, we have $\fkp_i + \fkp_j \in \Spec A$ for all $1 \le i,j \le \ell$. Hence, by Proposition \ref{3.2} we get the following.

\begin{cor}\label{3.3}
$\rmV(I) = \bigcup_{i\ne j}\rmV(\fkp_i + \fkp_j)$. Consequently, $\height_AI= \min \{\height_A(\fkp_i+\fkp_j) \mid  i \ne j\}$.
\end{cor}

\begin{prop}
$A$ is an unmixed ring if and only if $|F_i|$ is independent of the choice of  $i$. When this is the case,  $|F_i \setminus F_j|=|F_j \setminus F_i|$ for all $i \ne j$ and $\height_AI$ = $\min \{|F_i \setminus F_j| \mid i \ne j \}.$
\end{prop}

\begin{proof}
Since $\dim A/\fkp_i= n - |F_i|$ for every $1 \le i \le \ell$, the first assertion follows. Suppose $A$ is unmixed and let $i \ne j$. We set $Q= \fkp_i + \fkp_j$. We then have
\begin{eqnarray*}
\height_AQ&=& \height_{A/\fkp_i}Q/\fkp_i = \height_{T/P_i}[(P_i + P_j)/P_i]\\
&=&|(F_i \cup F_j) \setminus F_i|\\
&=&|F_j \setminus F_i|.
\end{eqnarray*}
For the same reason, $\height_AQ = \height_{A/\fkp_j}Q/\fkp_j = |F_i \setminus F_j|$, whence the second assertion follows. The last assertion is now clear.
\end{proof}

We now obtain the following.

\begin{thm}
Suppose that $|F_i|$ is independent of the choice of $i$ and that $|F_i \setminus F_j| \ge 2$ for all $i \ne j$. Let $d = \dim A$. Then the following assertions hold true.
\begin{enumerate}[$(1)$]
\item $\height_A I \ge 2$.
\item $d= n - |F_i| \ge 2$.
\item $B = I:I$, $\depth_AB=d$, and $\rmK_A \cong B$ as an $A$-module.
\item $0 < \depth A < d$.
\end{enumerate}
\end{thm}

\begin{proof}
The assertions (1), (2) are clear. Notice that $B= \bigoplus_{i=1}^\ell A/\fkp_i$ and we get the exact sequence $0 \to A \to B \to C \to 0$ of $A$-modules, where $IC=(0)$. Hence, Assertions (3), (4) follow from Corollary \ref{prop2}. 
\end{proof}

In the following subsections, we shall study the question of which value among $\{1,2, \ldots, d-1 \}$ the invariant $t= \depth A$ can vary.

\subsection{The case where $\depth A = 1$}\label{buchsbaum}
Let $m,n$ be integers such that $m \ge 4$, $n \ge 6$, and $\frac{2}{3}n \ge m \ge \frac{1}{2}n + 1$. We choose $F_1, F_2, F_3$ so that
\begin{eqnarray*}
F_1&=& \{1,2,\ldots, m \}\\
F_2&=&\{n-m+1, n-m+2, \ldots, n \}\\
F_3&=&\{m+1, m+2, \ldots, n\} \cup \{1,2, \ldots, 2m-n\}
\end{eqnarray*}
where for simplicity, we identify the number $i$ with the indeterminate $X_i$.
Hence, $|F_i|=m$ for all $i$, $F_i \not\subseteq F_j$ if $i \ne j$, and
\begin{center}
$|F_2 \setminus F_1|=|F_3 \setminus F_1|=n-m$, \   $|F_3 \setminus F_2|=2m-n$.
\end{center}
Therefore, $\height_AI=2m-n, \dim A= n-m$, and $\dim A/I=2n-3m$. Hence, $I=\m$ if and only if $2n=3m$, where $\m$ denotes the maximal ideal of $A$. Let $x_i$ denote the image of $X_i$ in $A$. We then have $$I = \fkp_2 \cap \fkp_3 + \fkp_1 \cap \fkp_3 + \fkp_1 \cap \fkp_2 = (x_1, \ldots, x_{2m-n})+(x_{n-m+1}, \ldots, x_m) + (x_{m+1}, \ldots,x_n),$$
so that $A/I$ is a regular local ring. Because $$I=IB=(I+\fkp_1)/\fkp_1 \oplus I/\fkp_2 \oplus I/\fkp_3$$ and $I+\fkp_1=\m$, we get $B/I = A/\m \oplus A/I \oplus A/I$,  whence $\depth_AB/I=0$ and $B/I$ is a regular ring. Since $B/A=(B/I)/(A/I)$, we have the following.

\begin{prop}\label{3.6}
$\depth_AB/A=0$ and hence $\depth A=1$.
\end{prop}

For simplicity, suppose $2n=3m$. Then, writing $m=2q$ and $n = 3q$ with $q \ge 2$, we get
\begin{eqnarray*}
F_1&=&\{1,2, \ldots,2q\},\\
F_2&=&\{q+1, q+2, \ldots, 3q\},\\
F_3&=&\{2q+1, 2q+2, \ldots, 3q\} \cup \{1, 2, \ldots, q\},
\end{eqnarray*}
$I = \m$, $\dim A = q$, and $B/A \cong (A/\m)^{\oplus 2}$ as an $A$-module. Therefore
$$\rmH_\m^1(A) \cong (A/\m)^{\oplus 2}\ \  \text{and}\ \ \rmH_\m^i(A) = (0)\ \ \text{for}\ \ i \ne 1, q,$$
so that $A$ is a Buchsbaum ring with $\depth A = 1$  (\cite[Proposition 2.12]{SV}).

For example, let $n =6$, $m=4$. Hence, $F_1=\{1,2,3,4\}$, $F_2=\{3,4,5,6\}$,  and $F_3=\{5,6, 1,2 \}$. Let $x_i$ denote the image of $X_i$ in $A$ and set $a= x_1+x_3+x_5$, $b=x_2+x_4+x_6$. Let $Q=(a,b)$. Then, $\m =QB$, and $a,b$ is a system of parameters of $A$.  For all $N \ge 1$ the Rees algebra $\calR_A(Q^N)$ is a Cohen-Macaulay ring of dimension $3$ (see \cite{GS2}). However, $\rmr(\calR_A(Q))=2$ and $\rmr(\calR_A(Q^N))= 2N-2$ for $N \ge 2$, so that $\calR_A(Q^N)$ is not a Gorenstein ring for any $N \ge 1$, where $\rmr(*)$ denotes the Cohen-Macaulay type. See \cite{GI} for more details.

\subsection{The case where $\depth A = 2$}
Let $\ell, m \ge 2$ be integers and set $n =\ell m$. Let $k$ be a field and let $T=k[[X_{ij} \mid 1 \le i \le \ell, 1 \le j \le m]]$ be the formal power series ring with $\ell m$ indeterminates $\{X_{ij}\}_{1 \le i \le \ell, 1 \le j \le m}$ over $k$. We set $F_i=\{X_{ij} \mid 1 \le j \le m\}$ for each $1 \le i \le \ell$. 
Hence, $\height_AI=m$ and $\dim A= m(\ell -1)$.

More concretely, let $\ell = 3$, $m \ge 2$, and consider the matrix 
$\left[\begin{smallmatrix}
X_1&X_2&\ldots&X_m\\
Y_1&Y_2&\ldots&Y_m\\
Z_1&Z_2&\ldots&Z_m\\
\end{smallmatrix}\right]
$ of indeterminates. Then, $\height_AI=m$, $\dim A = 2m$, and $\dim A/I=m$. We have 
\begin{eqnarray*}
I&=&\fkp_2 \cap \fkp_3 + \fkp_1 \cap \fkp_3 + \fkp_1 \cap \fkp_2\\
&=&(\fkp_1 + \fkp_2 \cap \fkp_3) + (\fkp_2 + \fkp_1 \cap \fkp_3) \cap (\fkp_3 + \fkp_1 \cap \fkp_2)\\
&=&(I + \fkp_1)\cap(I + \fkp_2) \cap (I +\fkp_3)\\
&=&(\fkp_1+\fkp_2)\cap(\fkp_1 + \fkp_3)\cap(\fkp_2 + \fkp_3).
\vspace{-2em}
\end{eqnarray*}
Therefore
{\footnotesize 
$$A/(I + \fkp_1) \cong T/(P_1 + P_2 \cap P_3) \cong k[[Y_1, Y_2, \ldots, Y_m, Z_1, Z_2, \ldots, Z_m]]/[(Y_1, Y_2, \ldots, Y_m) \cap (Z_1, Z_2, \ldots, Z_m)],
$$
}
so that $\depth A/(I + \fkp_1)=1$. By symmetry, $\depth A/(I + \fkp_i) =1$ for all $i$, whence $\depth_AB/I=1$, because $B/I= \bigoplus_{i=1}^3A/(I + \fkp_i)$. On the other hand, since $$I= \left[(P_1 + P_2) \cap (P_1 + P_3) \cap (P_2 + P_3) \right]/\fka$$ where $\fka = \bigcap_{i=1}^3 P_i$, our ring $A$ is exactly one of the special case $2n=3m$ of the previous subsection \ref{buchsbaum}, so that $A/I$ is a Buchsbaum ring with $\rmH_m^1(A/I) =(A/\m)^{\oplus 2}$.

We set $C = B/A$ and claim the following.

\begin{claim}\label{claim}
$\rmH_\m^1(C) \ne (0)$. Hence, $\depth_AC= 0$ or $1$, and $\depth A = 1$ or $2$.
\end{claim}

In fact, apply the functor $\rmH_\m^i(*)$ to the exact sequence $0 \to A/I \to B/I \to C \to 0$, and consider the long exact sequence
$$
0 \to \rmH_m^0(C) \to \rmH_m^1(A/I) \to \rmH_m^1(B/I) \to \rmH_m^1(C)
$$
of local cohomology modules. We then have $\rmH_m^1(C) \ne (0)$, because $\ell_A(\rmH_m^1(A))=2$ and $\ell_A(\rmH_m^1(B/I)) \ge 3$. We however eventually get the following.

\begin{prop}\label{3.7}
$\dim A = 2m \ge 4$, $\depth A = 2$, and $\widetilde{A}=B$ is a regular ring.
\end{prop}

\begin{proof}
Since $X_1 + Y_1 \not\in \bigcup_{i=1}^3P_i$, $X_1+Y_1$ is $A$-regular, and 
$$
A/(X_1+Y_1)A \cong k[[\{X_i\}_{2 \le i \le m}, \{Y_i\}_{1 \le i \le m}, \{Z_i\}_{1 \le i \le m}]]/J,
$$
where $$J= (Y_1, X_2, \ldots, X_m) \cap (Y_1, Y_2, \ldots, Y_m) \cap (Y_1^2, X_2 \ldots, X_m, Y_2, \ldots, Y_m) \cap (Z_1, Z_2, \ldots, Z_m).$$ Therefore, $\depth A \ge 2$, since $\depth A/(X_1+Y_1)A \ge 1$. Hence, $\depth A=2$ and $\depth_AC=1$ by Claim \ref{claim}. 
\end{proof}

\subsection{The case where $\depth A \ge 3$}
Let us now construct the examples of $A$ such that $\depth A \ge 3$. Let $q,m$ be integers such that $3 \le q < m$ and set $n = 2m$. We choose $F_1, F_2, F_3$ so that
$F_1 = \{1,2,\ldots, m \}$, $F_2 = \{q, q+1, \ldots, m+q-1 \}$, and $F_3 = \{m+1, m+2, \ldots, n\}$. Hence, $|F_i|=m$ for all $i$, and
\begin{center}
$|F_1 \setminus F_2|=q-1,  \ |F_2 \setminus F_3|=m-q+1$, \   $|F_3 \setminus F_1|=m$.
\end{center}
Therefore, $\height_AI=\min \{m-q+1, q-1\}$ and $\dim A= m$. Let $x_i$ denote the image of $X_i$ in $A$. Then\begin{eqnarray*}
I&=& \fkp_1 \cap \fkp_2 + \fkp_2 \cap \fkp_3 + \fkp_3 \cap \fkp_1\\
&=& (x_q, x_{q+1}, \ldots, x_{m+q-1})+(x_{1},x_2, \ldots, x_{q-1}){\cdot}(x_{m+q}, x_{m+q+1}, \ldots,x_n)\\
&=&(x_1, x_2, \ldots, x_{m+q-1}) \cap (x_q, x_{q+1}, \ldots, x_n),\\
\vspace{-2em}
\end{eqnarray*}
while
$I + \fkp_1 =(x_1, x_2, \ldots, x_{m+q-1})$, $I \supseteq \fkp_2$, and $I + \fkp_3= (x_q, x_{q+1},\ldots, x_n)$, 
so that $$B/I = A/[I + \fkp_1] \oplus A/I \oplus A/[I + \fkp_3].$$
  We set $C = B/A$. Then, the canonical embedding $A/I \to B/I$ is a split monomorphism of $A$-modules, and we get $C =A/[I + \fkp_1] \oplus A/[I + \fkp_3]$. Consequently
$$\dim_AC= \max \{m-q+1, q-1\}, \ \ \depth_AC = \min \{m-q+1, q-1\},$$ and thanks to the exact sequences
$0 \to A \to B \to C \to 0$ and 
$$0 \to A/I \to A/(x_1, x_2, \ldots, x_{m+q-1}) \oplus A/(x_q, x_{q+1}, \ldots, x_n) \to A/\m \to 0,$$ we get the following.

\begin{prop}\label{3.6}
$\depth A = \min \{m-q+1, q-1\} + 1 \ge 3$ and $\depth A/I = 1$.
\end{prop}


\begin{thebibliography}{20}

\bibitem{AF}
{\sc M. D'Anna and M. Fontana}, An amalgamated duplication of a ring along an ideal: the basic properties, {\em J. Algebra Appl.}, {\bf 6} (2007), 443--459.

\bibitem{A}{\sc Y. Aoyama}, Some basic results on canonical modules, {\em J. Math. Kyoto Univ.}, {\bf 23} (1983), 85--94.

\bibitem{AG}{\sc Y. Aoyama and S. Goto}, On the endomorphism ring of the canonical module, {\em J. Math. Kyoto Univ.}, {\bf 25} (1985), 21--30. 

\bibitem{G1}{\sc S. Goto}, On the Cohen-Macaulayfication of certain Buchsbaum rings, {\em Nagoya Math. J.}, {\bf  80} (1980), 107--116. 

\bibitem{G2}{\sc S. Goto}, On Buchsbaum rings obtained by gluing, {\em Nagoya Math. J.}, {\bf 83} (1981), 123--135. 

\bibitem{GS}{\sc S. Goto and Y. Shimoda}, On the Rees algebras of Cohen-Macaulay local rings, Commutative algebra (Fairfax, Va., 1979), 201--231, Lecture Notes in Pure and Appl. Math., {\bf 68}, Dekker, New York, 1982.

\bibitem{GS2}
{\sc S. Goto and Y. Shimoda}, On Rees algebras over Buchsbaum rings, {\em J. Math. Kyoto Univ.}, {\bf 20} (1980), 691--708. 

\bibitem{GI}
{\sc S. Goto and S. Iai}, Gorensteinness in Rees algebras of powers of parameter ideals, arXiv:2112.06676.

\bibitem{EGI}{\sc S. Goto, R. Isobe, and N. Taniguchi}, Ulrich ideals and 2-AGL rings, {\em J. Algebra}, {\bf 555} (2020), 96--130.

\bibitem{HK}
{\sc J. Herzog and E. Kunz}, Der kanonische Modul eines Cohen-Macaulay-Rings, Lecture Notes in Mathematics, 238, {\em Springer-Verlag, Berlin-New York}, 1971.

\bibitem{Hoa}
{\sc L. T. Hoa}, Reduction numbers and Rees algebras of powers of an Ideal, {\em Proc. Amer. Math. Soc.}, {\bf 119} (1993), no. 2, 415--422. 

\bibitem{kawasaki}
{\sc T. Kawasaki}, On arithmetic Macaulayfication of Noetherian rings, {\em Trans. Amer. Math. Soc.}, {\bf 354} (2001), 123--149.

\bibitem{Lindo}
{\sc H. Lindo}, Self-injective commutative rings have no nontivial rigid ideals, arXiv:1710.01793.

\bibitem{Ratliff}{\sc L. J.  Ratliff Jr.}, On quasi-unmixed local domains, the altitude formula, and the chain condition for prime ideals, (I), {\em Amer. J. Math.}, {\bf 91} (1969), no. 2, 508--528.

\bibitem{SV}
{\sc J. St\"uckrad and W. Vogel}, Buchsbaum rings and applications. An interaction between algebra, geometry and topology. Springer-Verlag, Berlin, 1986.

\end{thebibliography}
\end{document}